\documentclass[11pt]{article}
\usepackage{amsmath, amsthm, amssymb, amsfonts, mathrsfs, mathtools}
\usepackage[dvips]{graphicx}
\newtheorem{theorem}{Theorem}[section]
\newtheorem{cor}[theorem]{Corollary}

\newtheorem{lem}[theorem]{Lemma}
\newtheorem{prop}[theorem]{Proposition}

\newtheorem{ex}{Example}[section]

\def \Zl {{\mathbb Z}}
\def \Nl {{\mathbb N}}
\def \Rl {{\mathbb R}}
\def \Ql {{\mathbb Q}}
\def \Cl {{\mathbb C}}

\def \x {{\mathbf x}}

\def \o {{\mathbf 0}}

\title{Pretty Good State Transfer on Some NEPS}

\author{ Hiranmoy Pal\\
Department of Mathematics\\
Indian Institute of Technology Guwahati\\
Guwahati, India - 781039\\
Email: hiranmoy@iitg.ernet.in\\
\\
Bikash Bhattacharjya\\
Department of Mathematics\\
Indian Institute of Technology Guwahati\\
Guwahati, India - 781039\\
Email: b.bikash@iitg.ernet.in
}

\begin{document}
\maketitle

\vspace{-0.3in}

\begin{center}{Abstract}\end{center}
Let $G$ be a graph with adjacency matrix $A$. The transition matrix of $G$ relative to $A$ is defined by $H_{A}(t):=\exp{(-itA)},\;t\in\Rl$. We say that the graph $G$ admits perfect state transfer between the verteices $u$ and $v$ at $\tau\in\Rl$ if the $uv$-th entry of $H_{A}(\tau)$ has unit modulus. Perfect state transfer is a rare phenomena so we consider an approximation called pretty good state transfer. We find that NEPS (Non-complete Extended P-Sum) of the path on three vertices with basis containing tuples with hamming weights of both parities do not exhibit perfect state transfer. But these NEPS admit pretty good state transfer with an additional condition. Further we investigate pretty good state transfer on Cartesian product of graphs and we find that a graph can have PGST from a vertex $u$ to two different vertices $v$ and $w$.


\section{Introduction}

We consider continuous-time quantum walk relative to the adjacency matrix of a graph. Perfect state transfer in quantum networks was introduced in \cite{bose}. Let $G$ be a graph with adjacency matrix $A$. The transition matrix of $G$ relative to $A$ is defined by
\[H_{A}(t):=\exp{(-itA)}=\sum\limits_{n\geq 0}(-i)^nA^n\frac{t^n}{n!},\;t\in\Rl.\] We say that the graph $G$ exhibits perfect state transfer between the verteices $u$ and $v$ at $\tau\in\Rl$ if there is $\gamma\in\Cl$ with $|\gamma|=1$ such that $H_{A}(\tau)e_{u}=\gamma e_{v}$. We are mainly interested in finding graphs having perfect state transfer. In \cite{chr2}, we find that Cartesian powers of $P_2$ which is the path on two vertices, and Cartesian powers of $P_3$ which is the path on three vertices, admit perfect state transfer. Some other related results are also given in \cite{ple}. Further the results has been generalized for the path $P_2$ in \cite{ber,che}. The question of existence of perfect state transfer in NEPS of $P_3$ was asked in \cite{stev}. In \cite{pal}, we find that when the hamming weight of each element of the basis of an NEPS of $P_3$ is of the same parity then an additional condition guarantees perfect state transfer in the NEPS. In this paper, we show that if we allow the basis set to contain elements with hamming weights of both parities, then there is no perfect state transfer in the NEPS. Some more results regarding perfect state transfer can be found in \cite{mil4,god1,god3,pal1,pal2}. \par
Perfect state transfer is a rare phenomena. So we consider an approximation to it which is known as pretty good state transfer. We say that a graph $G$ has pretty good state transfer (PGST) from a vertex $u$ to $v$ if, for each $\epsilon>0$ there exists $t\in\Rl$ such that
\begin{center}
$||e_u^T H(t)e_v|-1|<\epsilon,$
\end{center}
\emph{i.e,} the modulus of the $uv$-th entry of the transition matrix comes arbitrary close to $1$. There are very few graphs known to have PGST. In \cite{god4}, Godsil et. al. found that the path $P_n$ admits PGST between the end vertices if and only if $n+1=2^m$, or if $n+1=p$ or $2p$ for some odd prime $p$. A characterization of PGST in double stars have been given in \cite{fan}. Some interesting results regarding entries of the transition matrix are also given in \cite{kirk}.\par
In this paper, we find a class of NEPS of $P_3$ exhibiting PGST where there is no perfect state transfer. Further we study Cartesian product of graphs and find a way to construct more graphs allowing PGST. As a corollary we obtain a class of NEPS with factor graphs $P_2$ and $P_3$ admitting PGST.\par
Now we define NEPS (Non-complete Extended P-Sum) of $n$ graphs $G_{1}, \ldots, G_{n}$ with a basis set $\Omega\subset\Zl_{2}^{n}\setminus\left\lbrace\o\right\rbrace$. The NEPS \cite{cev} of $G_{1}, \ldots, G_{n}$ with basis $\Omega$ is denoted by $NEPS\left(G_{1}, \ldots, G_{n};\Omega\right)$. This NEPS has the vertex set $V\left(G_{1}\right)\times\cdots\times V\left(G_{n}\right)$. Two vertices $(x_{1}, \ldots, x_{n})$ and $(y_{1}, \ldots, y_{n})$ are adjacent in $NEPS\left(G_{1}, \ldots, G_{n}; \Omega\right)$ if and only if there is an $n$-tuple $\left(\beta_{1}, \ldots, \beta_{n}\right)\in\Omega$ such that $x_{i}=y_{i}$ exactly when $\beta_{i}=0$ and $x_{i}$ is adjacent to $y_{i}$ in $G_{i}$ exactly when $\beta_{i}=1$. Suppose the graphs $G_{1}, \ldots,G_{n}$ have the adjacency matrices $A_{1}, \ldots, A_{n}$, respectively. Then the adjacency matrix of $NEPS\left(G_{1}, \ldots,G_{n}; \Omega\right)$ can be obtained as \[A_{\Omega} = \sum\limits_{\beta\in\Omega}A_{1}^{\beta_{1}}\otimes \cdots\otimes A_{n}^{\beta_{n}},\]
where $A\otimes B$ denotes the tensor product of two matrices $A$ and $B$. Suppose the number of vertices in $G_i$ is $n_i$. For $i=1,2,\ldots,n$, let $G_i$ have the eigenvalues $\lambda_{i1},\ldots,\lambda_{in_i}$, not necessarily distinct and suppose the corresponding eigenvectors are $\x_{i1},\ldots,\x_{in_i}$. Then the NEPS with basis $\Omega$ has the eigenvalues
\begin{eqnarray}\label{q0}
\Lambda_{j_1\ldots j_n}=\sum\limits_{\beta\in\Omega}\lambda_{1j_1}^{\beta_1}\cdots\lambda_{nj_n}^{\beta_n}, \; j_k=1,\ldots,n_k,\; k=1,\ldots,n.
\end{eqnarray}
Here the eigenvector corresponding to $\Lambda_{j_1\ldots j_n}$ is the vector $\x_{1j_1}\otimes\cdots\otimes\x_{nj_n}$. See \cite{cev} for details.

\section{No perfect state transfer in a class of NEPS of $P_3$}

Suppose $G$ is a graph with transition matrix $H_{A}(t)$ relative to the adjacency matrix $A$.  The graph is said to be periodic at a vertex $u$ if there is $\gamma\in\Cl$ with $|\gamma|=1$ and $\tau(\neq 0)\in\Rl$ so that \[H_{A}(\tau)e_u=\gamma e_u.\] The graph is said to be periodic if it is periodic at all vertices at the same time. In such a case we have $H_A(\tau)=\gamma I$ where $I$ is the identity matrix of appropriate order. It is well known that periodicity is necessary for a graph to exhibit perfect state transfer which follows from the following Lemma.

\begin{lem}\cite{god3}\label{p0}
If a graph $G$ admits perfect state transfer from a vertex $u$ to another vertex $v$ at time $\tau$ then $G$ is periodic at $u$ and $v$ with period $2\tau$.
\end{lem}

We are dealing with undirected simple graphs so the associated adjacency matrices are symmetric. Assume that $A$ is the adjacency matrix of a graph $G$. Suppose $\lambda_1,\ldots,\lambda_m$ are the distinct eigenvalues of $A$ and let the projections to the respective eigenspaces be $E_1,\ldots,E_m$. Then the transition matrix of $G$ can be evaluated as
\[H_A(t)=\exp{(-itA)}=\sum\limits_{r=1}^m \exp{(-it\lambda_r)}E_r.\] 

Suppose the graph $G$ has $n$ vertices. The eigenvalue support of a vector $\x\in\Rl^n$ is a set containing those eigenvalues $\lambda_r$ of $A$ for which $E_r \x\neq\o$. Here we include one simple observation. Suppose $E_{r}=F_1+\cdots+F_k$ where $F_iF_j=\delta_{ij}F_i$ and $F_i^T=F_i$, $1\leq i,j\leq k$. Here $\delta_{ij}$ is the usual Kronecker delta function. Now for any vector $\x$, we obtain
\[\left(F_i \x\right)^TE_r \x=\left(F_i \x\right)^T\left(F_1\x+\cdots+F_k\x\right)=||F_i \x||^2,\; 1\leq i\leq k.\] 
Therefore if $F_i \x\neq\o$ for some $i$ then the eigenvalue $\lambda_r$ belongs to the eigenvalue support of the vector $\x$. The following theorem characterizes the eigenvalues of a periodic graph.

\begin{theorem}\cite{god3}\label{p1}
Let a graph $G$ be periodic at a vertex $u$. If $\lambda_k, \lambda_l, \lambda_r, \lambda_s$ are eigenvalues in the eigenvalue support of $e_u$ and $\lambda_r\neq\lambda_s$ then
\begin{eqnarray}\label{q1}
\frac{\lambda_k-\lambda_l}{\lambda_r-\lambda_s}\in\Ql.
\end{eqnarray}
\end{theorem}
Assume that $\Omega\subset\Zl_{2}^{n}\setminus\left\lbrace\o\right\rbrace$ which contains tuples with hamming weights of both parities. Now we find that NEPS of $P_3$ with basis $\Omega$ does not exhibit perfect state transfer. We denote $s(\beta)$ to be the hamming weight of a tuple $\beta$. 

\begin{theorem}\label{p2}
Suppose $\Omega_e$ and $\Omega_o$ are both nonempty subsets of $\Zl_{2}^{n}\setminus\left\lbrace\o\right\rbrace$ such that the hamming weight of each tuple in $\Omega_e$ and $\Omega_o$ are even and odd, respectively. If $\Omega=\Omega_e\cup\Omega_o$ then $NEPS\left(P_3, \ldots, P_3;\Omega\right)$ does not exhibit perfect state transfer.
\end{theorem}

\begin{proof}
The graph $P_{3}$ has the eigenvalues $-\sqrt{2},\;0,\;\sqrt{2}$ and the corresponding eigenvectors are 
\begin{eqnarray*}
       \begin{array}{lll}
     \x_{1} = \left(\begin{array}{r}
               1\\
               -\sqrt{2}\\
               1\end{array} \right), &
     \x_{2} =\left(\begin{array}{r}
               1\\
               0\\
               -1\end{array} \right), &
     \x_{3} =\left(\begin{array}{r}
               1\\
               \sqrt{2}\\
               1\end{array} \right). \end{array}
\end{eqnarray*}
Therefore the eigenvectors of $NEPS\left(P_3, \ldots, P_3;\Omega\right)$ are
\[\x_{j_1}\otimes\cdots\otimes\x_{j_n},\; \text{where }j_1,\ldots,j_n\in\left\lbrace 1,\;2,\;3\right\rbrace.\]
Notice that, if $j_1,\ldots,j_n\in\left\lbrace 1,\;3\right\rbrace$ then all columns of $\left(\x_{j_1}\otimes\cdots\otimes\x_{j_n}\right)\left(\x_{j_1}\otimes\cdots\otimes\x_{j_n}\right)^T$ are nonzero. Hence, using the observation preceding Theorem \ref{p1}, the eigenvalue support of all standard unit vectors in $\Rl^{3^n}$ contains the eigenvalue corresponding to the vector $\x_{j_1}\otimes\cdots\otimes\x_{j_n}.$
\par Observe that the hamming weight $s(\beta)$ is even or odd according as $\beta$ is in $\Omega_e$ or $\Omega_o$. Using equation (\ref{q0}), the eigenvalue corresponding to $\x_1\otimes \x_1\otimes\cdots\otimes \x_1$ can be calculated as
\begin{eqnarray*}
\sum\limits_{\beta\in\Omega}\left(-\sqrt{2}\right)^{s(\beta)}
& = &\sum\limits_{\beta\in\Omega_e}\left(-\sqrt{2}\right)^{s(\beta)}+\sum\limits_{\beta\in\Omega_o}\left(-\sqrt{2}\right)^{s(\beta)}\\
& = &\sum\limits_{\beta\in\Omega_e}\left(\sqrt{2}\right)^{s(\beta)}-\sum\limits_{\beta\in\Omega_o}\left(\sqrt{2}\right)^{s(\beta)}\\
& = & a-b\sqrt{2},\; \text{ where $a$ and $b$ are positive integers}.
\end{eqnarray*}
Similarly the eigenvalue corresponding to the vector $\x_3\otimes \x_3\otimes\cdots\otimes \x_3$ can be obtained in terms of $a$ and $b$ as
\begin{eqnarray*}
\sum\limits_{\beta\in\Omega}\left(\sqrt{2}\right)^{s(\beta)}
& = &\sum\limits_{\beta\in\Omega_e}\left(\sqrt{2}\right)^{s(\beta)}+\sum\limits_{\beta\in\Omega_o}\left(\sqrt{2}\right)^{s(\beta)}\\
& = & a+b\sqrt{2}.
\end{eqnarray*}
For $1\leq j\leq n$ consider $\Omega_e ^j= \left\lbrace \beta\in\Omega_e:\beta_j=1\right\rbrace$. Since $\Omega_e$ is a non-empty subset of $\Zl_{2}^{n}\setminus\left\lbrace\o\right\rbrace$ there is at least one $j$ for which $\Omega_e ^j$ is also non-empty. Without loss of generality let $j=1$ and hence the integral part of the eigenvalue corresponding to $\x_1\otimes \x_3\otimes\cdots\otimes \x_3$ is
\[\sum\limits_{\beta\in\Omega_e^1}\left(-\sqrt{2}\right)\left(\sqrt{2}\right)^{s(\beta)-1}
+\sum\limits_{\beta\in\Omega_e\setminus\Omega_e^1}\left(\sqrt{2}\right)^{s(\beta)},\]
which is clearly not equal to the integer $a$. Assume that the eigenvalue corresponding to the vector $\x_1\otimes \x_3\otimes\cdots\otimes \x_3$ is $c+d\sqrt{2}$, for some integers $c(\neq a)$ and $d$.\par
Now if $NEPS\left(P_3, \ldots, P_3;\Omega\right)$ exhibits perfect state transfer then by Lemma \ref{p0} the graph is periodic at some vertex, say $u$. Hence by Theorem \ref{p1} the eigenvalues in the eigenvalue support of the characteristic vector of $u$ must satisfy the ratio condition (\ref{q1}). The eigenvalues $a+b\sqrt{2},\; a-b\sqrt{2}\text{ and }c+d\sqrt{2}$ lies in the eigenvalue support of $e_u$. Note that
\begin{eqnarray*}
\frac{\left(a+b\sqrt{2}\right)- \left(a-b\sqrt{2}\right)}{\left(a+b\sqrt{2}\right)-\left(c+d\sqrt{2}\right)}
& = & \frac{2b\sqrt{2}}{\left(a-c\right)+\left(b-d\right)\sqrt{2}}\\
& = & \frac{2b\sqrt{2}\left[\left(a-c\right)-\left(b-d\right)\sqrt{2} \right]}{\left(a-c\right)^2-2\left(b-d\right)^2}\not\in\Ql.
\end{eqnarray*}
Therefore $NEPS\left(P_3, \ldots, P_3;\Omega\right)$ is not periodic at all vertices and hence the graph does not exhibit perfect state transfer.
\end{proof}

The above result gives a partial characterization of perfect state transfer in the class of all NEPS of $P_3$. In the following section we investigate PGST on NEPS of $P_3$ with basis $\Omega$ containing tuples with hamming weights of both parities.

\section{Pretty good state transfer on NEPS of $P_{3}$}

Before discussing PGST on NEPS we mention some previously known results. The following result shows that transition matrix of an NEPS can be realized as a product of transition matrices of some of its spanning subgraphs. 

\begin{prop}\cite{pal}\label{a1}
Let $G_{1}, \ldots,G_{n}$ be $n$ graphs and consider $\Omega\subset\Zl_{2}^{n}\setminus\left\lbrace\o\right\rbrace$. For $\beta\in\Omega$, let $H_{\beta}(t)$ be the transition matrix of $NEPS\left(G_{1}, \ldots,G_{n}; \left\lbrace\beta\right\rbrace\right)$. Then $NEPS\left(G_{1}, \ldots,G_{n}; \Omega\right)$ has the transition matrix \[H_{\Omega}(t)=\prod\limits_{\beta\in\Omega}H_{\beta}(t).\]
\end{prop}

Let us denote $\tau_{k}=\frac{\pi}{\left(\sqrt{2}\right)^{k}}$ for $k\in\Nl$. Then we have the following lemma.

\begin{lem}\cite{pal}\label{a2}
Let $\beta=\left(\beta_{1},\ldots \beta_{n}\right)\in\Zl_{2}^{n}\setminus\left\lbrace\o\right\rbrace$ and suppose $H_{\beta}(t)$ is the transition matrix of $NEPS\left(P_{3}, \ldots,P_{3}; \left\lbrace\beta\right\rbrace\right)$. If the hamming weight of $\beta$ is $k$ then $H_{\beta}(-\tau_{k})=H_{\beta}(\tau_{k})$.
\end{lem}

Now we find the transition matrix of an NEPS of $P_3$ at a specific time depending on the basis of the NEPS.

\begin{theorem}\cite{pal}\label{a3}
Consider $\Omega\subset\Zl_{2}^{n}\setminus\left\lbrace\o\right\rbrace$ such that for all $\beta\in\Omega$, the number $s(\beta)$ is even (or odd). Let $k=\min\limits_{\beta\in\Omega} s(\beta)$ and $\Omega^{*}=\left\lbrace\beta\in\Omega : s(\beta)=k \right\rbrace $. If the transition matrices of NEPS of $P_{3}$ corresponding to $\Omega$ and $\Omega^{*}$ are $H_{\Omega}(t)$ and $H_{\Omega^{*}}(t)$, respectively, then $H_{\Omega}(\tau_{k})=H_{\Omega^{*}}(\tau_{k})$. 
\end{theorem}
In the following result we find that an NEPS with basis containing tuples having hamming weights odd (or even) is essentially periodic.
\begin{theorem}\label{a4}
Suppose $\Omega\subset\Zl_{2}^{n}\setminus\left\lbrace\o\right\rbrace$ satisfies the conditions of Theorem \ref{a3}. Then the graph $NEPS\left(P_{3},\ldots,P_{3};\Omega\right)$ is periodic at $2\tau_k$.
\end{theorem}
\begin{proof}
Using Lemma \ref{a2}, for $\beta\in\Omega^*$, we find that $H_{\beta}(-\tau_{k})=H_{\beta}(\tau_{k})$. Further, Proposition \ref{a1} implies that
\begin{eqnarray*}
H_{\Omega^*}(-\tau_k)
& = &\prod\limits_{\beta\in\Omega^*}H_{\beta}(-\tau_k)\\
& = &\prod\limits_{\beta\in\Omega^*}H_{\beta}(\tau_k)\\
& = & H_{\Omega^*}(\tau_k).
\end{eqnarray*}
Therefore we obtain $H_{\Omega^*}(2\tau_k)=I$. Now Theorem \ref{a3} gives
\[H_{\Omega}(2\tau_{k})=\left(H_{\Omega}(\tau_{k})\right)^2=\left(H_{\Omega^{*}}(\tau_{k})\right)^2=H_{\Omega^*}(2\tau_k)=I.\]
Hence $NEPS\left(P_{3},\ldots,P_{3};\Omega\right)$ is periodic at $2\tau_k$.
\end{proof}
Note that if $k$ is even then $\tau_{k}=\frac{\pi}{\left(\sqrt{2}\right)^{k}}=\frac{\pi}{2^{k/2}}$. Recall that if a graph is periodic at $\tau$ then it is periodic at $2^r\tau$ for all non-negative integer $r$. As an implication, we have the following obvious corollary which will be used to find PGST in some NEPS of $P_3$.
\begin{cor}\label{a5}
Let $\Omega\subset\Zl_{2}^{n}\setminus\left\lbrace\o\right\rbrace$ be such that the hamming weight of $\beta$ for each $\beta\in\Omega$ is even. Then the NEPS of $P_{3}$ corresponding to $\Omega$ is periodic at $\pi$.
\end{cor}
Consider an NEPS with basis $\Omega$. Let $M(\Omega)$ be the matrix formed by writing the vectors in $\Omega$ as its rows. We denote the rank of $M(\Omega)$ over $\Zl_{2}$ by $r(\Omega)$.
\begin{theorem}\cite{stev1}\label{a6}
Let $B_{1},\ldots, B_{n}$ be connected bipartite graphs. Then $NEPS(B_{1},\ldots, B_{n}; \Omega)$ is connected if and only if the rank $r(\Omega)=n$.
\end{theorem}

The following result gives a sufficient condition for an NEPS of $P_{3}$ to be connected and exhibit perfect state transfer. Since we are interested in connected graphs we impose the condition $r(\Omega)=n$ in the following result only to ensure that the graphs involved are connected.

\begin{theorem}\cite{pal}\label{a7} Consider $\Omega\subset\Zl_{2}^{n}\setminus\left\lbrace\o\right\rbrace$ such that $r(\Omega)=n$. Also, for $\beta\in\Omega$, assume that $s(\beta)$ is even (or odd). Suppose that $k=\min\limits_{\beta\in\Omega} s(\beta)$ and $\Omega^{*}=\left\lbrace\beta\in\Omega : s(\beta)=k \right\rbrace $. If $\sum\limits_{\beta\in\Omega^{*}}\beta\neq \mathbf{0}$ in $\Zl_{2}^{n}$ then $NEPS\left(P_{3},\ldots,P_{3};\Omega\right)$ allows perfect state transfer at time $\frac{\pi}{(\sqrt{2})^{k}}$.
\end{theorem}
Consider the following approximation theorem by Kronecker which plays a crucial role in finding PGST in the later results.

\begin{theorem}[Kronecker's approximation theorem]\cite{apo}\label{c1}
Let $\theta$ be an irrational number and suppose $\alpha$ is a real number. For every $\delta>0$ there are integers $p$ and $q$ such that
\[|p\theta-q-\alpha|<\delta.\]
\end{theorem}

Suppose $\Omega_e$ and $\Omega_o$ are non-empty subsets of $\Zl_{2}^{n}\setminus\left\lbrace\o\right\rbrace$ containing tuples with hamming weights even and odd, respectively.
\begin{theorem}\label{c2}
Let $\Omega=\Omega_e\cup\Omega_o\subset\Zl_{2}^{n}\setminus\left\lbrace\o\right\rbrace$ be such that both $\Omega_e$, $\Omega_o$ are non-empty and $r(\Omega)=n$. Assume that $k=\min\limits_{\beta\in\Omega_e} s(\beta)$ and $l=\min\limits_{\beta\in\Omega_o} s(\beta)$ with $\Omega_e^{*}=\left\lbrace\beta\in\Omega_e : s(\beta)=k \right\rbrace $ and $\Omega_o^{*}=\left\lbrace\beta\in\Omega_o : s(\beta)=l \right\rbrace $. Then pretty good state transfer occurs in $NEPS\left(P_{3},\ldots,P_{3};\Omega\right)$ if any one of the following conditions holds:
\begin{enumerate}
\item  $\sum\limits_{\beta\in\Omega_o^{*}}\beta\neq \mathbf{0}$ in $\Zl_{2}^{n}$, or
\item $\sum\limits_{\beta\in\Omega_e^{*}}\beta\neq \mathbf{0}$ in $\Zl_{2}^{n}$.
\end{enumerate}
\end{theorem}
\begin{proof}
Proposition \ref{a1} implies that $H_{\Omega}(t)=H_{\Omega_e}(t)H_{\Omega_o}(t)$. As $k$ is even, by Corollary \ref{a5}, we find that $NEPS\left(P_{3},\ldots,P_{3};\Omega_e\right)$ is periodic at $\pi$ and hence for all integer $q$ we have $H_{\Omega_e}(q\pi)=I$. Let us denote $\tau=\frac{\pi}{(\sqrt{2})^{l}}$ and $\eta=\frac{\pi}{(\sqrt{2})^{k}}$.\\
\textbf{Case I:} Assume that $\sum\limits_{\beta\in\Omega_o^{*}}\beta\neq \mathbf{0}$ in $\Zl_{2}^{n}$. By Theorem \ref{a7}, the graph $NEPS\left(P_{3},\ldots,P_{3};\Omega_o\right)$ admits perfect state transfer at $\tau$, say, between the pair of vertices $u$, $v$, \emph{i.e}, $|e_u^TH_{\Omega_o}(\tau)e_v|=1$. Let us consider
\[f(t)=|e^T_uH_{\Omega_o}(t)e_v|,\]
which is necessarily a continuous function. By Theorem \ref{a4}, the function $f(t)$ is also periodic with period $2\tau$ and therefore $f(t)$ is uniformly continuous on $\Rl$. So for $\epsilon>0$, there exist $\delta>0$ such that $|t-t'|<\delta$ implies $|f(t)-f(t')|<\epsilon$. Consider $\alpha=\frac{1}{(\sqrt{2})^{l}}$ and $\theta=2\alpha$. Since $l$ is an odd number the chosen number $\theta$ is indeed an irrational number. By Theorem \ref{c1}, for $\delta>0$ there exists integers $p$ and $q$ such that $|p\theta-q-\alpha|<\frac{\delta}{\pi}$, \emph{i.e}, $|(2p-1)\alpha-q|<\frac{\delta}{\pi}$. Now $|(2p-1)\tau-q\pi|<\delta$ implies that
\begin{center}
$|f\left((2p-1)\tau\right)-f(q\pi)|<\epsilon$.
\end{center}
Now notice that $f\left((2p-1)\tau\right)=f(\tau)=1$ and hence $|f(q\pi)-1|<\epsilon$. Therefore we obtain
\begin{center}
$H_{\Omega}(q\pi)=H_{\Omega_e}(q\pi)H_{\Omega_o}(q\pi)=H_{\Omega_o}(q\pi),$
\end{center}
which in turn implies that $f(q\pi)=|e_u^TH_{\Omega}(q\pi)e_v|$. Thus for $\epsilon>0$ there exists $q\pi\in\Rl$ such that 
\begin{center}
$||e_u^TH_{\Omega}(q\pi)e_v|-1|<\epsilon.$
\end{center}
So $NEPS\left(P_{3},\ldots,P_{3};\Omega\right)$ has PGST between the pair of vertices $u$ and $v$.\\
\textbf{Case II:} Suppose $\sum\limits_{\beta\in\Omega_e^{*}}\beta\neq \mathbf{0}$ in $\Zl_{2}^{n}$. By Theorem \ref{a7}, the graph $NEPS\left(P_{3},\ldots,P_{3};\Omega_e\right)$ exhibits perfect state transfer at $\eta$, say, between the pair of vertices $u$ and $v$. Consider
\[g(t)=|e^T_uH_{\Omega_e}(t)e_v|.\]
By the same argument as in Case I, the function $g(t)$ is uniformly continuous. Therefore for $\epsilon>0$, there exist $\delta>0$ so that $|t-t'|<\delta$ implies $|g(t)-g(t')|<\epsilon$. Notice that for every integer $q$, we have $g(q\pi+\eta)=g(\eta)=1$. Consider $\theta=\frac{2}{(\sqrt{2})^{l}}$, which is an irrational number as $l$ is an odd natural number. Also let $\alpha=\frac{1}{\left(\sqrt{2}\right)^k}$. Then by Theorem \ref{c1}, for $\delta>0$ there exists integers $p$ and $q$ such that $|p\theta-q-\alpha|<\frac{\delta}{\pi}$. Now $|2p\tau-(q\pi+\eta)|<\delta$ implies that
\begin{center}
$|g\left(2p\tau\right)-g(q\pi+\eta)|<\epsilon$, \emph{i.e}, $|g(2p\tau)-1|<\epsilon$.
\end{center}
Finally we have \[H_{\Omega}(2p\tau)=H_{\Omega_e}(2p\tau)H_{\Omega_o}(2p\tau)=H_{\Omega_e}(2p\tau),\;\text{as } H_{\Omega_o}(2p\tau)=I.\]
So for each $\epsilon>0$ there exists $2p\tau\in\Rl$ so that $||e_u^TH_{\Omega}(2p\tau)e_v|-1|<\epsilon$. Hence PGST occurs between the pair of vertices $u$ and $v$ in $NEPS\left(P_{3},\ldots,P_{3};\Omega\right)$.
\end{proof}

Thus we find an infinite class of graphs allowing PGST. In the following section we find some other graphs exhibiting PGST.

\section{Pretty good state transfer on Cartesian products}

The Cartesian product of two graphs $G_{1}$ and $G_{2}$ with vertex sets $V_{1}$ and $V_{2}$ is the graph $G_{1}\square G_{2}$, with vertex set $V_{1}\times V_{2}$. Two vertices $(u_{1},u_{2})$ and $(v_{1},v_{2})$ are adjacent in $G_{1}\square G_{2}$ if and only if either $u_{1}$ is adjacent to $v_{1}$ in $G_{1}$ and $u_{2}=v_{2}$, or $u_{1}=v_{1}$ and $u_{2}$ is adjacent to $v_{2}$ in $G_{2}$. The transition matrix of a Cartesian product of two graphs is given by the following result.
\begin{lem}\cite{chr2}
Let $G_1$ and $G_2$ be two graphs and suppose $G_1\square G_2$ is the Cartesian product of $G_1$ and $G_2$. If $G_1$ and $G_2$ have the transition matrices $H_{G_1}(t)$ and $H_{G_2}(t)$, respectively, then the transition matrix of $G_1\square G_2$ is\[H_{G_1\square G_2}(t)=H_{G_1}(t)\otimes H_{G_2}(t).\]
\end{lem}
Now we investigate PGST on Cartesian product of two graphs. 
\begin{theorem}\label{f}
Let $G_1$ and $G_2$ be two graphs so that $G_1$ is periodic at a vertex at $\tau$ and $G_2$ exhibits perfect state transfer at $\eta$. If $\tau$ and $\eta$ are independent over the rational numbers then $G_1\square G_2$ admits pretty good state transfer.
\end{theorem}
\begin{proof}
Since $\tau$ and $\eta$ are independent over $\Ql$, by Kronecker's approximation theorem the set $\left\lbrace m\tau-2n\eta\; : \;m,n\in\Zl \right\rbrace$ is dense in $\Rl$. Hence for $\delta>0$, there exists $m,n\in\Zl$ so that
\begin{center}
$|m\tau-(2n+1)\eta|<\delta$.
\end{center}
Suppose $G_1$ is periodic at the vertex $u$ at $\tau$, \emph{i.e,} $|e_{u}^T H_{G_1}(m\tau)e_{u}|=1$ for all integer $m$. Also assume that $G_2$ exhibits perfect state transfer at $\eta$ between the pair of vertices $v$ and $w$. Consider $f(t)=|e_{v}^TH_{G_2}(t)e_{w}|$. For each integer $n$ we have
\begin{eqnarray*}
f((2n+1)\eta)
&=& |e_{v}^TH_{G_2}\left((2n+1)\eta\right)e_{w}| \\
&=& |e_{v}^TH_{G_2}\left(2n\eta\right)H_{G_2}\left(\eta\right)e_{w}| \\
&=& |e_{v}^TH_{G_2}\left(2n\eta\right)\gamma e_{v}|,\;\text{as } H_{G_2}\left(\eta\right)e_{w}=\gamma e_{v},\text{ for some }\gamma\in\Cl \text{ with } |\gamma|=1\\
&=& |e_{v}^TH_{G_2}\left(2n\eta\right)e_{v}|.
\end{eqnarray*}
By Lemma \ref{p0}, the graph $G_2$ is periodic at the vertex $v$ at $2\eta$ and therefore \[1=|e_{v}^TH_{G_2}\left(2n\eta\right)e_{v}|=f((2n+1)\eta).\]
Now $f(t)$ is uniformly continuous and therefore for $\epsilon>0$ there is $m,n\in\Zl$ so that
\begin{center}
$|f\left(m\tau\right)-f\left((2n+1)\eta\right)|<\epsilon$, \emph{i.e}, $|f\left(m\tau\right)-1|<\epsilon.$
\end{center}
Since $H_{G_1\square G_2}(t)=H_{G_1}(t)\otimes H_{G_2}(t)$, using properties of tensor product, we obtain
\begin{eqnarray*}
|\left(e_{u}\otimes e_{v}\right)^T\left(H_{G_1\square G_2}(m\tau)\right)\left(e_{u}\otimes e_{w}\right)|
& = & |\left(e_{u}\otimes e_{v}\right)^T\left(H_{G_1}(m\tau)\otimes H_{G_2}(m\tau)\right)\left(e_{u}\otimes e_{w}\right)|\\
& = & |\left(e_{u}^T H_{G_1}(m\tau)e_{u}\right)\otimes\left(e_{v}^T H_{G_2}(m\tau)e_{w}\right)|\\
& = & f(m\tau).
\end{eqnarray*}
Hence $G_1\square G_2$ has PGST between the pair of vertices $\left(u,v\right)$ and $\left(u,w\right)$.
\end{proof}
It is interesting to see that a graph can have PGST from a vertex $u$ to two different vertices $v$ and $w$. However, it is well known that if $v\neq w$ then there cannot be perfect state transfer from $u$ to $v$ and also from $u$ to $w$. Consider the following example.
\begin{ex}
The path $P_2$ with vertices, say $u,\;v$, exhibits perfect state transfer between the pair of vertices $u$ and $v$ at $\frac{\pi}{2}$. It is also well known that $P_2$ is periodic at $\pi$. On the other hand the path $P_3$ with vertices, say $1,\;2$ and $3$, admits perfect state transfer between the pair of vertices $1$ and $3$ at $\frac{\pi}{\sqrt{2}}$. The path $P_3$ is also periodic at $\frac{2\pi}{\sqrt{2}}$. Hence by Theorem \ref{f}, the Cartesian product $P_2\square P_3$ allows PGST between the pair of vertices $(u,1)$ and $(u,3)$. Again we find that $P_3\square P_2$ admits PGST between the vertices $(1,u)$ and $(1,v)$. Since $P_3\square P_2\cong P_2\square P_3$, we see that $P_2\square P_3$ also exhibits PGST between the pair of vertices $(u,1)$ and $(v,1)$.
\end{ex}
Now using Theorem \ref{f}, we find that some of the NEPS with factor graphs $P_2$ and $P_3$ which can be realized as a Cartesian product of an NEPS of $P_2$ and an NEPS of $P_3$, exhibits PGST. 
  
\begin{cor}
Let $\Omega\subset\Zl_{2}^{m}\setminus\left\lbrace\o\right\rbrace$ and $\Omega'\subset\Zl_{2}^{n}\setminus\left\lbrace\o\right\rbrace$. Suppose the hamming weight of each tuple in $\Omega$ is odd, $k=\min\limits_{\beta\in\Omega} s(\beta)$ and $\Omega^{*}=\left\lbrace\beta\in\Omega : s(\beta)=k \right\rbrace $. If $\sum\limits_{\beta\in\Omega^*}\beta\neq \o$ or $\sum\limits_{\beta\in\Omega'}\beta\neq \o$ then the Cartesian product of $NEPS\left(P_3,\ldots,P_3;\Omega\right)$ and $NEPS\left(P_2,\ldots,P_2;\Omega'\right)$ exhibits PGST.
\end{cor}
\begin{proof}
Assume that $\sum\limits_{\beta\in\Omega^*}\beta\neq \o$ in $\Zl_{2}^{m}$. Then by Theorem \ref{a7}, $NEPS\left(P_3,\ldots,P_3;\Omega\right)$ admits perfect state transfer at $\frac{\pi}{\left(\sqrt{2}\right)^k}$. Further, Lemma $7.2$ in \cite{god1} implies that $NEPS\left(P_2,\ldots,P_2;\Omega'\right)$ is periodic at $\pi$. Since $k$ is odd, the numbers $\pi$ and $\frac{\pi}{\left(\sqrt{2}\right)^k}$ are independent over $\Ql$. Therefore, by applying Theorem \ref{c2}, we see that the Cartesian product of $NEPS\left(P_2,\ldots,P_2;\Omega'\right)$ and $NEPS\left(P_3,\ldots,P_3;\Omega\right)$ exhibits PGST. Hence  we have PGST in the Cartesian product of $NEPS\left(P_3,\ldots,P_3;\Omega\right)$ and $NEPS\left(P_2,\ldots,P_2;\Omega'\right)$.\par
Suppose $\sum\limits_{\beta\in\Omega'}\beta\neq \o$ in $\Zl_{2}^{n}$. Then by Lemma $7.2$ in \cite{god1}, the graph $NEPS\left(P_2,\ldots,P_2;\Omega'\right)$ admits perfect state transfer at $\frac{\pi}{2}$. Also, by Theorem \ref{a4}, we find that $NEPS\left(P_3,\ldots,P_3;\Omega\right)$ is periodic at $\frac{2\pi}{\left(\sqrt{2}\right)^k}$. Hence, by Theorem \ref{c2}, the Cartesian product of $NEPS\left(P_3,\ldots,P_3;\Omega\right)$ and $NEPS\left(P_2,\ldots,P_2;\Omega'\right)$ exhibits PGST.
\end{proof}
\section*{Conclusions}
In  \cite{pal}, we have seen that there are many NEPS of $P_3$ allowing perfect state transfer. Here we considered the case when the basis of an NEPS of $P_3$ contains tuples with hamming weights of both parities. We found that in such cases there is no perfect state transfer. Then we looked for PGST in that class of graphs. We found that many of them admit PGST.
\par The only class of NEPS of $P_3$ where we still do not know whether there is perfect state transfer are those NEPS in Theorem \ref{a7} where $\sum\limits_{\beta\in\Omega^{*}}\beta= \mathbf{0}$ in $\Zl_{2}^{n}$. By Theorem \ref{a4}, we find that these graphs are essentially periodic. Therefore, if a graph in that class has PGST then it must also exhibit perfect state transfer. 
\par Finally, in Theorem \ref{f}, we found a general result on PGST on Cartesian product of graphs. Using that we have seen that there are NEPS with factor graphs $P_2$ and $P_3$ exhibiting PGST. Following this we can construct many other type of graphs allowing PGST.   

\section*{Acknowledgment}
We sincerely thank Prof. Sukanta Pati and Prof. Anjan K. Chakrabarty of Indian Institute of Technology Guwahati for giving useful insights into the material presented.  

\end{document}